\documentclass{article}
\usepackage{amsmath,amsxtra,amsthm,amssymb,graphics,mathrsfs,amscd}
\usepackage[top=25mm, bottom=25mm, left=25mm, right=25mm]{geometry}
\usepackage[all]{xy}
\usepackage{enumerate}
\usepackage{float}
\usepackage{mathtools}
\usepackage{parskip, setspace}
\usepackage[pdftex]{hyperref}
\hypersetup{
  bookmarksnumbered = true,
  bookmarksopen = true,
  bookmarksopenlevel = 2,
  colorlinks = false,
  pageanchor=true,
  pdfcreator={Aaron Chan},
  pdfstartview = FitH
}

\numberwithin{equation}{section}

\begingroup
    \makeatletter
    \@for\theoremstyle:=definition,remark,plain\do{%
        \expandafter\g@addto@macro\csname th@\theoremstyle\endcsname{%
            \addtolength\thm@preskip\parskip
            }%
        }
\endgroup

\theoremstyle{plain}
\newtheorem{theorem}{Theorem}

\newtheorem{prop}[theorem]{Proposition}
\newtheorem{lemma}[theorem]{Lemma}
\newtheorem{corollary}[theorem]{Corollary}
\newtheorem{definition}[theorem]{Definition}

\numberwithin{theorem}{section}

\theoremstyle{definition}
\newtheorem{example}{Example}
\newtheorem{remark}[example]{Remark}

\makeatletter
\def\uddots{\mathinner{\mkern1mu\raise\p@
\vbox{\kern7\p@\hbox{.}}\mkern2mu
\raise4\p@\hbox{.}\mkern2mu\raise7\p@\hbox{.}\mkern1mu}}
\makeatother

\mathchardef\hy="2D

\newcommand{\integer}{\mathbb{Z}}

\newcommand{\isom}{\cong}

\newcommand{\Hom}{\mathrm{Hom}}
\newcommand{\sthom}{\underline{\mathrm{Hom}}}
\newcommand{\End}{\mathrm{End}}

\newcommand{\catC}{\mathcal{C}}
\newcommand{\catT}{\mathcal{T}}
\newcommand{\stmod}{\mathrm{\underline{mod}}\hy}

\newcommand{\rmod}{\mathrm{mod}\hy}

\newcommand{\proj}[1]{\mathrm{proj}\hy #1}
\newcommand{\add}{\mathrm{add}}
\newcommand{\filt}{\mathcal{F}}

\DeclareMathOperator{\Soc}{soc}
\DeclareMathOperator{\Rad}{rad}

\newcommand{\config}{\mathcal{C}}

\newcommand{\sms}{\mathcal{S}}

\newcommand{\twotilt}{\mathsf{2tilt}}
\newcommand{\stinv}[1]{\smash{\overline{#1}}\vphantom{#1}^{-1}}

\newcommand{\smss}{\mathsf{sms}}

\newcommand{\silt}{\mathsf{silt}}
\newcommand{\tilt}{\mathsf{tilt}}
\newcommand{\twosilt}{\mathsf{2silt}}

\newcommand{\frakf}{\mathfrak{F}}
\newcommand{\BrTree}{\mathsf{BrTree}}

\begin{document}
\title{Two-term tilting complexes and simple-minded systems of self-injective Nakayama algebras}
\author{Aaron Chan}
\maketitle

\begin{abstract}
We study the relation between simple-minded systems and two-term tilting complexes for self-injective Nakayama algebras.  More precisely, we show that any simple-minded system of a self-injective Nakayama algebra is the image of the set of simple modules under a stable equivalence, which is given by the restriction of a standard derived equivalence induced by a two-term tilting complex.  We achieve this by exploiting and connecting the mutation theories from the combinatorics of Brauer tree, configurations of stable translations quivers of type A, and triangulations of a punctured convex regular polygon.

\end{abstract}

\section{Introduction}
\label{sec:2tilt-1}
Constructing derived equivalences and stable equivalences are fundamental problems in representation theory.  For finite dimensional self-injective algebras, many stable equivalences are induced by a corresponding derived equivalence.  In \cite{CKL}, we have surveyed and studied such connection in details for representation-finite self-injective (RFS) algebras, i.e. basic indecomposable non-simple algebras for which the module category consists only of finitely many indecomposable modules up to isomorphism.  This connection was established via the notion of \emph{simple-minded system} (sms).  Loosely speaking, such a system generates the stable module category in the same way as the set of simple modules.  In fact, we have found that simple-minded systems of RFS algebras are precisely the image of the set of simple modules under a stable equivalence.  We ask if one can find a description on how the \emph{finite} set of simple-minded systems control the \emph{infinite} set of tilting complexes.

Recent advances \cite{AI} in the tilting theory of derived categories suggest that one can systematically construct tilting complexes (which give rise to derived equivalences) by a technique called \emph{mutations}.  For RFS algebras, it is known from \cite{Aiha2,CKL} that every tilting complexes can be obtained by iterative mutations.  For simple-minded systems, there also exists a theory of mutation.  Moreover, the two mutations are compatible with each other, so we can pick out the stable equivalence we want by tracking through the mutations used to obtain the relevant tilting complex.  On the other hand, it has been observed in \cite{AH} that tilting complexes of RFS algebras are given by ``compositions" of two-term tilting complexes, i.e. tilting complexes which concentrated in two consecutive degrees.  This motivates us to compare simple-minded systems and two-term tilting complexes.  In this article, we initiate such a study on a subclass of RFS algebras, namely the uniserial ones, aka the self-injective Nakayama algebras.

Let us be more precise now.  We will work with a finite dimensional self-injective $k$-algebra throughout, where $k$ is an algebraically closed field.
For an RFS algebra $A$, let $\tilt(A)$ (resp. $\twotilt(A)$) denote the set of basic (resp. two-term) tilting complexes concentrated in non-positive degrees, up to homotopy equivalence and shifts.  We denote by $\smss(A)$ the set of simple-minded systems of $A$, and by $\sms_A$ the set of (isoclass representatives of) simple modules of $A$.  There is a natural injection from $\twotilt(A)$ to $\tilt(A)$; this also gives an injection of the corresponding exchange quivers.  From the result of \cite{CKL}, there is a surjection from $\tilt(A)$ to $\smss(A)$, which also gives a surjection on the corresponding exchange quivers.  We are interested in the composition $\frakf$ of these two maps.
We investigate the case when $A$ is Nakayama, i.e. uniserial.
We denote $A_n^\ell$ as the self-injective Nakayama algebra with $n$ simples and of Loewy length $\ell+1$.
For a tilting complex $T$, we will denote $E_T$ the endomorphism algebra of $T$, $F_T:D^b(\rmod{A})\to D^b(\rmod{E_T})$ the corresponding standard derived equivalence \cite{Rickard1991} and $\overline{F_T}:\stmod{A}\to \stmod{E_T}$ the restriction of $F_T$ on the stable categories.  The composition $\frakf$ we are interested in is the map $T\mapsto \stinv{F_T}(\sms_{E_T})$.
The following is our main result.
\begin{theorem}\label{thm-main}
For a self-injective Nakayama algebra $A_n^\ell$ with $\ell\neq \gcd(n,\ell)$ (resp. $\ell=\gcd(n,\ell)$), the map $\frakf:\twotilt(A_n^\ell)\to\smss(A_n^\ell)$ given by $T\mapsto \stinv{F_T}(\sms_{E_T})$ is a bijection (resp. surjection non-bijection).  Moreover, for an indecomposable summand $X$ of $T\in\twotilt(A_n^\ell)$, $\frakf(\mu_X(T))=\mu_S^-(\frakf(T))$ for some (unique) $S\in\frakf(T)$.  In particular, when $\ell\neq\gcd(n,\ell)$, the exchange quiver of $\twotilt(A_n^\ell)$ embeds into that of $\smss(A_n^\ell)$.
\end{theorem}

Let $A_n^n$ (resp. $A_n^{nk}$) be the symmetric Nakayama algebra with $n$ simples and Loewy length $n+1$ (resp. $nk+1$ for any $k>1$).  The relations between the sets considered in this article are shown in the diagram below.

\[
\xymatrix@R=30pt@C=35pt{
& \twotilt(A_n^{nk}) \ar@{_{(}->}[ld] \ar@{->>}@/^4.2pc/[dddd]|{\mbox{\cite{RS}}} \ar@{<->}[dd]^{\ref{thm-main}} & \mathcal{T}(n)\times\{\pm1\} \ar@{<->}[r]^{\mbox{\cite{AIR,Adachi}}}\ar@{<->}[l]_{\mbox{\cite{AIR,Adachi}}}\ar@{->>}@/^2.5pc/[ldddd]|{\ref{prop-triangualtion-to-BTree}}\ar@{->>}@/_2.5pc/[rdddd]|{\ref{prop-triangualtion-to-BTree}} & \twotilt(A_n^n) \ar@{->>}@/_4.2pc/[dddd]|{\mbox{\cite{RS}}} \ar@{^{(}->}[rd] \ar@{->>}[dd]^{\ref{thm-main}} & \\
\tilt(A_n^{nk})\ar@{->>}[rd]_{\mbox{\cite{CKL}}}\ar@{->>}@/_3.5pc/[rddd]|{\mbox{\cite{Ric2}}}\ar@{<->}@/^2pc/[rrrr] & & & & \tilt(A_n^n)\ar@{->>}[ld]^{\mbox{\cite{CKL}}}\ar@{->>}@/^3.5pc/[lddd]|{\mbox{\cite{Ric2}}} \\
& \smss(A_n^{nk})\ar@{<->}[d]_{\mbox{\cite{CKL}}}\ar@{->>}@/^1pc/[rr] & &  \smss(A_n^n)\ar@{<->}[d]^{\mbox{\cite{CKL}}} & \\
& \mathsf{Conf}(\integer A_{nk}/\langle \tau^n\rangle)\ar@{->>}[d]_{\mbox{\cite{GR,Riedtmann2}}}\ar@{->>}@/^1pc/[rr]|{\ref{prop-surject-config}} & & \mathsf{Conf}(\integer A_n)\ar@{->>}[d]^{\mbox{\cite{GR,Riedtmann2}}} & \\
& \BrTree(n,k)\ar@{->>}@/^1pc/[rr] & & \mathsf{BrTree}(n,1) &
}
\]
In the above picture, we have denoted the set of triangulations on a punctured $n$-disc (Definition \ref{defn-triangulation}) as $\mathcal{T}(n)$, the set of configurations (Definition \ref{def-combinatorial-configuration})
of the translation quiver $\integer A_n$ as $\mathsf{Conf}(\integer A_n)$, the set of Brauer trees (Definition \ref{defn-naka-alg}) with $n$ edges and multiplicity $k$ as $\BrTree(n,k)$.  Also note that all the maps shown preserve mutations on respective sets.

The following connections from the sets shown in the above diagram are worth noting.  In \cite{Adachi}, the author showed a bijection between the set of $n$-part compositions of $n$ with $\mathcal{T}(n)$.  This set is of particular interest to, for example, representations theory coming from enumerative combinatorics and Lie theory.  In \cite{AIR}, it was shown that the set $\twotilt(\Lambda)$ corresponds to the set of functorially finite torsion classes of $\rmod{\Lambda}$.  In \cite{BLR}, a corollary of the main theorem is that configurations of $\integer A_n$ corresponds to (multiplicity-free) tilting modules of the quiver algebra $kA_n$.  In \cite{Reading}, it is shown that configurations of $\integer A_n$ can be interpreted as non-crossing partitions, which appear in many other contexts in representation theory and combinatorics.  It will be interesting to find implications of our result on these areas.

A corollary of the main theorem is that every simple-minded system can be obtained from a derived equivalence given by a two-term tilting complex.  This result is not apparent from the definition of $\frakf$, and is false even for RFS algebras in general.
\begin{corollary}
Let $A$ be a self-injective Nakayama algebra and $F:D^b(\rmod{A})\to D^b(\rmod{B})$ be a derived equivalence.  Then there is a two-term tilting complex $T\in\twotilt(A)$ which gives rise to the following diagram.
$$
\xymatrix@R=8pt{
\stmod{E_T}\; \ar[r]^{\overline{F_T}} & \stmod{A} & \;\stmod{B} \ar[l]_{\overline{F}}\\
\text{\llap{\{simple $E_T$}-modules\}} \ar@{|->}[r] & \quad\sms\quad & \text{\{simple \rlap{$B$-modules\}}} \ar@{|->}[l]
}
$$
where $\overline{F_T}$ and $\overline{F}$ are the induced stable equivalences and $n\in\integer$.  In particular, $B\isom E_T$.
\end{corollary}
\begin{proof}
This follows from combining Theorem \ref{thm-main}, and Linckelmann's theorem \cite{Linckelmann1996}. (See also Theorem B and C from \cite{CKL}.)
\end{proof}

In the next section \ref{sec:prelim}, we will go through definitions of the objects of interest in this article.  We will also present some lemmas required for the proof of the main theorem.  We interpret a combinatorial observation, originally from \cite{Riedtmann3}, as a reduction scheme to study simple-minded systems of self-injective Nakayama algebras.  As noted before, a simple-minded system for a RFS algebra comes from the image of simple modules under some stable equivalence.  For the same stably equivalent algebra, one can find different stable equivalences giving different simple-minded systems.  The interpretation we found will help us distinguish between these simple-minded systems, which is necessary for tracking changes of simple-minded system under mutations.

In section \ref{sec:2tilt-mutations}, we will study the mutations used in this article.  Then we will introduce a combinatorial description of two-term silting complexes for self-injective Nakayama algebras found by Adachi \cite{Adachi}.  We will use this to obtain a combinatorial description of two-term tilting complexes via rotational-symmetric triangulations on a punctured disc, refining Adachi's result.  We will use mutation of the  Brauer tree \cite{KZ,Kauer,Aiha} to obtain a sequence of mutations needed to acquire a two-term tilting complex from the original algebra.  Our sequence of mutations guarantees that at each mutation step, the mutated tilting complex stays in $\twotilt(A_n^\ell)$, and so refines a result of \cite{KZ}.  We then give some detailed observations on how simple-minded systems change under mutation.  Finally, we present the proof of the main theorem \ref{thm-main} in section \ref{sec:2tilt-proof}.

\section{Preliminaries}\label{sec:prelim}
\subsection{Sms's, configurations, silting complexes}
Let $\stmod{A}$ denote the stable module category of an algebra $A$.  This is the category with the same objects as $\rmod{A}$ but with Hom-space $\sthom_A(X,Y):=\Hom_{\stmod{A}}(X,Y)$ given by $\Hom_A(X,Y)$ modulo the morphisms which factor through projective modules.  Two algebras are \emph{stable equivalent} if their stable module categories are equivalent.  It is well-known that $\stmod A$ is a triangulated category if and only if $A$ is self-injective.  The suspsension functor for $\stmod{A}$ is the inverse syzygy $\Omega^{-1}$.  For any collections $\sms_1, \sms_2$ of objects in $\stmod{A}$, we define a collection of objects
\[
\sms_1\ast\sms_2 = \{ X\in\catT | \exists \mbox{distinguished triangle } S_1\to X\to S_2\to \Omega^{-1}S_1 \mbox{ with }S_1\in\sms_1, S_2\in\sms_2\}
\]
For a collection $\sms$ of objects in $\stmod{A}$, we denote $(\sms)_0=\{0\}$, and for $n\in\integer_{> 0}$ define inductively $(\sms)_n = (\sms)_{n-1} \ast \sms\cup\{0\}$.  Similarly, one can define ${}_n(\sms)$, but it can be shown that ${}_n(\sms) = (\sms)_n$ for any $n\geq 0$ \cite[Lemma 2.2]{D3}.  For a full subcategory $\catC$ of $\catT$ (we will always identify $\catC$ with the set of its objects), we say $\catC$ is \emph{extension closed} if $\catC\ast\catC \subset \catC$.  We define the \emph{filtration closure} (or \emph{extension closure}) of a collection $\sms$ of objects of $\stmod{A}$ as $\filt(\sms):=(\bigcup_{n\geq 0} (\sms)_n)$, which is the smallest extension closed full subcategory of $\stmod{A}$ containing $\sms$ \cite[Lemma 2.3]{D3}.

\begin{definition}\label{defn-sms}
A collection $\sms$ of objects in $\stmod{A}$ is a \emph{simple-minded system} of $A$ if $\filt(\sms)=\stmod{A}$, and $\sms$ is a system of pairwise orthogonal stable bricks, i.e.
\begin{equation}\label{eqn-sms-ortho-axiom}
\sthom_A(S,T) = \left\{\begin{array}{ll} 0 & (S\neq T), \\
k & (S=T).\end{array}\right.
\end{equation}
We denote by $\smss(A)$ the collection of all simple-minded systems of $A$.
\end{definition}
The definition presented above is taken from \cite{D3}.  This is different from the original definition of simple-minded system from \cite{KL}.  In \cite{KL}, simple-minded system is defined for stable module category of any artinian algebra, which does not necessarily possess any triangulated structure.  Nevertheless, we are only interested in stable module category of a self-injective algebra here.  In which case, the two definitions are compatible with each other.  We will abbreviate simple-minded system by \emph{sms} from now on.

\begin{definition}[\cite{Riedtmann2}]\label{def-combinatorial-configuration}
Let $\Gamma=\mathbb{Z}Q/\Pi$ be a stable translation quiver of tree class Dynkin quiver $Q$. A \emph{configuration} $\mathcal{C}$ is a set of vertices of $\Gamma$ which
satisfies the following conditions:
\begin{enumerate}
\item For any $e, f\in \mathcal{C}$,
$\Hom_{k(\Gamma)}(e,f)= \left\{\begin{array}{ll} 0 & (e\neq f), \\
k & (e=f).\end{array}\right.$
\item For any $e\in\Delta_0$, there exists some $f\in \mathcal{C}$ such that $\Hom_{k(\Gamma)}(e,f)\neq 0$.
\end{enumerate}
Here $k(\Gamma)$ is the mesh category of the stable translation quiver $\Gamma$.
\end{definition}
In this article, we will only look at Dynkin quiver of type $\mathbb{A}_n$.  We denote such a quiver $A_n$, which should not be confused with the notation for self-injective Nakayama algebras.

\begin{definition}[\cite{AI}]\label{defn-silting-obj}
Let $\catT=K^b(\proj{A})$ be the bounded homotopy category of complexes over finitely generated projective $A$-modules.  We call a complex $T\in\catT$ a \emph{silting} (resp. \emph{tilting}) if the smallest thick subcategory of $\catT$ containing $T$ is $\catT$ itself, and $\Hom_\catT(T,T[i])=0$ for any $i>0$ (resp. $i\neq 0$).
We denote the set of silting (resp. tilting) complexes over $A$ as $\silt(A)$ (resp. $\tilt(A)$).
\end{definition}
As noted in \cite{KY}, a silting complex is tilting if and only if it is Nakayama-stable (stable under the Nakayama functor).

\subsection{Self-injective Nakayama and Brauer tree algebras}
\begin{definition}[see for example \cite{GR,ARS,ASS}]\label{defn-naka-alg}
\begin{enumerate}
\item A \emph{self-injective Nakayama algebra} $A_n^\ell$ with $n$ simples and Loewy length $\ell+1$ is given by the path algebra $kQ/I$ with quiver
\begin{equation}\label{eqn-quiver}
Q:\quad \xymatrix@R=10pt{
& 1\ar[r]^{\alpha_1} & n \ar[rd]^{\alpha_n} &\\
2\ar[ur]^{\alpha_2} & & & n-1 \ar@/^1pc/@{.}[lll]
}
\end{equation}
and relation ideal $I=\Rad^{\ell+1}(kQ)$.

\item A \emph{Brauer graph} $G$ is a finite undirected connected graph (possibly with loops and multiple edges) with the following data.  To each vertex we assign a cyclic ordering of edges incident to it, and a positive integer called \emph{multiplicity}.

\item A \emph{Brauer tree} is a Brauer graph which is a tree, having at most one vertex with multiplicity greater than one.  If there is such vertex, it is called \emph{exceptional vertex}, otherwise we say the Brauer tree has trivial multiplicity.  Traditionally, we choose the counter-clockwise direction as the cyclic ordering of edges; and denote the Brauer tree as $(G,v,m)$ for a tree $G$ with exceptional multiplicity $m$ at the exceptional vertex $v$.  For simplicity, we usually just use $G$ as the notation for this triple.

\item A finite dimensional algebra $A$ is a \emph{Brauer tree algebra} associated to a given Brauer tree $(G,v,m)$, if there is a one-to-one correspondence between the edges $j$ of $G$ and the simple $A$-modules $S_j$ in such a way that the projective cover $P_j$ of $S_j$ has the following description.  We have $P_j/\Rad P_j\isom \Soc P_j \isom S_j$, and the heart $\Rad P_j/\Soc P_j$ is a direct sum of two (possibly zero) uniserial modules $U_j$ and $W_j$ corresponding to the two vertices $u$ and $w$ at the end of the edge $j$.  If the edges around $u$ are cyclically ordered $j,j_1,j_2,\ldots, j_r,j$ and the multiplicity of the vertex $u$ is $m_u$, then the corresponding uniserial module $U_j$ has composition factors (from the top) $S_{j_1}, S_{j_2}, \ldots,  S_{j_r}$, $S_j, S_{j_1}, \ldots, S_{j_r}$, $S_j, \ldots, S_{j_r}$ so that $S_{j_1},\ldots, S_{j_r}$ appear $m_u$ times and $S_j$ appears $m_u-1$ times.  We denote the basic algebra associated to a Brauer tree $G$ with $e$ edges and exceptional multiplicity $m$ as $B_{e,m}^G$.
\end{enumerate}
\end{definition}

A \emph{Brauer star}, which we usually denote by $\star$, is a Brauer tree where the underlying graph is a star, with exceptional vertex at the centre, the corresponding Brauer tree algebra is called Brauer star algebra.  Note that the class of Brauer star algebras coincides with the class of symmetric Nakayama algebras, i.e. $B_{e,m}^\star = A_e^{em}$ for any $e,m\geq 1$, and so we fix the quiver and relation presentation for Brauer star algebra with the one given by Nakayama algebra.

If $n$ is the number of simple modules for, we denote $\overline{i}$ to be the positive integer in $\{1,\ldots,n\}$ with $i\equiv \overline{i}$ mod $n$ for any $i\in\integer$.  Now the radical of projective indecomposable $P_i$ of a self-injective Nakayama algebra has projective cover $P_{\overline{i+1}}$.  All indecomposable $A_n^\ell$-modules are uniserial, hence uniquely determined by its socle and Loewy length, and so we denote $M_{i,l}$ with $i\in\{1,\ldots, n\}$ and $1 \leq l \leq \ell+1$.  For any $i$ and any $l \leq \ell$, the Auslander-Reiten translate $\tau\isom \nu_{A_n^\ell}\Omega^2$ sends $M_{i,l}$ to $M_{\overline{i+1},l}$, where $\nu_{A_n^\ell}$ is the Nakayama functor of $A_n^\ell$.  The Heller translate $\Omega$, which is the inverse of suspension functor in the triangulated category $\stmod{A_n^\ell}$, sends $M_{i,l}$ to $M_{\overline{i-l+1},\ell+1-l}$; and inverse Heller translate $\Omega^{-1}$ sends $M_{i,l}$ to $M_{\overline{i-l},\ell+1-l}$.  The Nakayama functor $\nu_{A_n^\ell}$ sends $M_{i,l}$ to $M_{\overline{i+e},l}$ where $e=\gcd(n,\ell)$.

From now on, we fix the coordination of the stable AR-quiver of $A_n^\ell$ using the pair appearing in the subscript of an indecomposable $A_n^\ell$-module.  Thus the simple $A_n^\ell$-modules lie on the bottom rim of the stable AR-quiver, and radical of projective indecomposable $A_n^\ell$-modules lie on the top rim of the stable AR-quiver.  Note that the stable AR-quiver ${}_s\Gamma_{A_n^\ell}$ is isomorphic to the stable tube $\integer A_\ell/\langle \tau^n\rangle$, 
so by \cite[3.6]{CKL}, we can identify $\smss(A_n^\ell)$ with the set of configurations of $\integer A_\ell/\langle \tau^n\rangle$ which are $\tau^{n\integer}$-stable.

We note that every configuration of $\integer A_n$ is $\tau^{n\integer}$-stable, i.e. configurations of $\integer A_n$ can be thought as configurations of $\integer A_n/\langle \tau^n\rangle$ (and vice versa) by taking modulo $n$ in the $x$-coordinate of the vertices.  In general, the configurations of $\integer A_\ell/\langle \tau^n\rangle$ correspond to the configurations of $\integer A_\ell/\langle \tau^e\rangle$ for $e=\gcd(\ell,n)$.  In particular, the AR-theory of $A_n^\ell$ can be thought as ``controlled" by the AR-theory of $A_e^\ell=A_e^{em}$, which is a Brauer tree algebra.  Due to the representation-finiteness nature of the algebras we work with, this means the (stable) module categories of self-injective Nakayama algebras are ``controlled" by that of Brauer tree algebras.  So one can very often lift a result for symmetric Nakayama algebras to self-injective Nakayama algebras.  Such technique are usually called \emph{covering theory}.

\begin{example}
The following is the stable AR-quiver of $B_{3,2}^\star = A_3^6$ (we omit the symbol $M$):
\[
\xymatrix@C=2pt@R=10pt{
  &  &  &  &  & (3,6) \ar[rd]& & (2,6) \ar[rd] & & (1,6)\ar[rd] & & (3,6) \\
  &  &  &  & (3,5) \ar[ru]\ar[rd]& & (2,5) \ar[ru]\ar[rd] & & (1,5) \ar[ru]\ar[rd] & & (3,5)\ar[ru]\\
  &  &  & (3,4) \ar[ru]\ar[rd]& & (2,4) \ar[ru]\ar[rd] & & (1,4) \ar[ru]\ar[rd] & & (3,4)\ar[ru]\\
  &  & (3,3) \ar[ru]\ar[rd]& & (2,3) \ar[ru]\ar[rd] & & (1,3) \ar[ru]\ar[rd] & & (3,3)\ar[ru]\\
  & (3,2) \ar[ru]\ar[rd]& & (2,2) \ar[ru]\ar[rd] & & (1,2) \ar[ru]\ar[rd] & & (3,2)\ar[ru]\\
(3,1) \ar[ru] & & (2,1) \ar[ru] & & (1,1) \ar[ru] & & (3,1)\ar[ru]
}
\]
$\{(1,1),(2,1),(3,1)\}$ is the set of simple $B_{3,2}^\star$-module.  Another example of simple-minded system (i.e. configuration) is $\sms=\{(1,1),(2,3),(3,5)\}$, the (unique) Brauer tree algebra $B$ such that $\sms$ is a $B$-simple-image is the one associated to the graph (tree) of a line with exceptional vertex at the end of the line, i.e. $\xymatrix@C=15pt{\bullet \ar@{-}[r]& \circ\ar@{-}[r] & \circ\ar@{-}[r] & \circ }$.
\end{example}
We will use this coordination throughout this article.  Note that our coordination is slightly different from the conventional one, where the ``x-axis'' goes from left to right, the transformation from our coordination to the conventional one is $(x,y)\mapsto (em-x,y)$.

\subsubsection{Sms's of Brauer tree algebras}\label{sec:2tilt-Btree1}
In the following, we play with the combinatorics of configurations to give us some tools which will be useful later in the proof of the theorem \ref{thm-main} in the symmetric case.  Recall that an \emph{extremal vertex} of Brauer tree $G$ is a vertex of valency 1; we call the edge connected to an extremal vertex as \emph{leaf}.

\begin{lemma}\label{lem-simpleLeaf}
Let $G$ be a Brauer tree and $S$ a simple $B_{e,m}^G$-module.  Then $S$ lies on the rim of the stable AR-quiver, if and only if, the edge in $G$ which corresponds to $S$ is a leaf attached to a non-exceptional extremal vertex.
\end{lemma}
\begin{proof}
By the construction of Brauer tree algebras, an edge is a leaf attached to a non-exceptional extremal vertex if and only if the corresponding indecomposable projective module is uniserial.  Also note that for any simple $B_{e,m}^G$-module $S$, whose projective cover is $P$, we have almost split sequence starting at $\Omega(S)$ being:
$$
0 \to \Omega(S) \to P\oplus \Rad(P)/\Soc(P) \to \Omega^{-1}(S) \to 0
$$
This says that $P$ is uniserial if and only if, $\Omega(S)$ and $\Omega^{-1}(S)$ is on a rim of the stable AR-quiver, say located (without loss of generality) at $(1,em), (e,em)$ respectively, which in turns is equivalent to $S$ located at $(1,1)$, i.e. another rim of the stable AR-quiver.
\end{proof}

We need some combinatorial lemma about configurations:
\begin{lemma}[\cite{Riedtmann3}, Lemma 2.5]
A set $\config$ in the vertex set $(\integer A_n)_0$ of $\integer A_n$ is a configuration of $\integer A_n$ if and only if 
$\omega_n\config \cup \tau^{(n+1)\integer}(0,1)$ is a configuration of $\integer A_{n+1}$, where $\omega_n$ is an injection $\omega_n:(\integer A_n)_0 \to (\integer A_{n+1})_0$ given by:
\begin{equation*}
\omega_n(x,y) = \begin{cases}
(x,y) & \text{if }0\leq x-y < x\leq n \\
(x,y+1) & \text{if }0<x<y
\end{cases}
\end{equation*}
and by the rule $ \omega_n\tau^n = \tau^{n+1}\omega_n$.
\end{lemma}

Using covering theory, we specialise this lemma for configurations of $\integer A_{em}$ which are stable under $\tau^{e\integer}$, one can see the lemma can be refined to the following.
\begin{lemma}
A set $\config$ of vertices in $(\integer A_{em}/\langle \tau^e\rangle)_0$ is a configuration of $\integer A_{em}/\langle \tau^e\rangle$ if and only if
\begin{equation*}
\config^+ := \omega_e^{(m)} \config \cup \{(e+1,1)\}
\end{equation*}
is a configuration of $\integer A_{(e-1)m}/\langle \tau^{e-1}\rangle$ where $\omega_e^{(m)}$ is an injection $\omega_e^{(m)}:(\integer A_{em})_0  \to  (\integer A_{(e+1)m})_0$ given by:
\begin{equation*}
\omega_e^{(m)}(x,y) =(x,y+\ell) 
\end{equation*}
when $-\ell e\leq x-y < (1-\ell)e$ for some $\ell\in \{0,\ldots, m\}$.
\end{lemma}
This gives a way to constructing a configuration of $\integer A_{em}/\langle\tau^e\rangle$ from the a configuration of $\integer A_{(e-1)m}/\langle \tau^{e-1}\rangle$ by ``inserting a vertex" at the bottom rim, or at the top rim by reflecting the mesh along a horizontal line in the middle.  
This in particular shows that any member of a configuration must have $y$-coordinate in $\{1,2,\ldots,e,e(m-1)+1,e(m-1)+2,\ldots,e(m-1)+e\}$.

We are going to show an observation on the reverse process of this ``insertion construction" of configurations.  This will ultimately help us distinguish different sms's obtained by different two-term tilting complexes.

Suppose we have a configuration $\config$ of $\integer A_{em}/\langle \tau^e\rangle$ with $m>1$, then by \cite[3.1, 3.6]{CKL} $\config$ corresponds to an image of simple $B_{e,m}^G$-module under some stable equivalence $F:\stmod{B_{e,m}^G}\to\stmod{B_{e,m}^\star}$.  Each $(x,y)\in\config$ lying on the rim (i.e. $(x,y)$ with $y=1$ or $em$) then corresponds to a leaf of $G$ by Lemma \ref{lem-simpleLeaf}.  Fix any one of such $(x,y)$, let $G'$ be the Brauer tree obtained by removing a leaf corresponding to $(x,y)$ in $G$.  Apply Heller shifts $\Omega^n$ (for some $n$) to $\config$ so that $(x,y)\mapsto (1,1)$.  More explicitly, apply $\tau^{-x+1}$ (i.e. $n=-2x+2$) if $y=1$, or $\tau^{-x+1}\Omega^{-1}$ (i.e. $n=-2x+1$) if $y=em$.  Now we have $\Omega^n\config = \mathcal{D}^+$ for some configuration $\mathcal{D}$ which is a $B_{e-1,m}^{G'}$-simple-image.  Obviously, $\Omega^{-n}\mathcal{D}$ is still a configuration of $B_{e-1,m}^{G'}$.  We call this process as ``\emph{cutting off a leaf}".  Being the reverse process of Riedtmanns's insertion construction, the corresponding effect of cutting of a leaf is reflected as removing the ``diagonals" going into and coming out of $(x,y)$ on the stable AR-quiver.  Repeating this process, one eventually reaches a stage when the remaining Brauer tree is a star with $h$ edges.  If $m>1$, there are exactly two possible truncated configurations left, namely $\config_{h,m}^-:=\{(i,1)|i=1,\ldots,h\}$ and $\config_{h,m}^+:=\{(i,\ell)|i=1,\ldots,h\}=\Omega \config_{h,m}^-$.  This tells us which $\tau$-orbit the original configuration $\config$ lies in.  To summarise:

\begin{corollary}[Tree Pruning Lemma]\label{cor-cutTree}
\index{tree pruning}
Suppose $G$ is a Brauer tree with $e$ edges and multiplicity $m$, where the valency of the exceptional vertex is $\ell$.  Let $\config$ be the configuration of $B_{e,m}^G$ representing the simple $B_{e,m}^G$-module.  Then the effect on $\config$ after successively cutting off leaves of $G$ until reaching $B_{h,m}^\star$ is either $\config_{h,m}^-$ or $\config_{h,m}^+$, depending only on the $\tau$-orbit for which $\config$ lies in.
\end{corollary}

By the virtue of this result, for Brauer tree with non-trivial exceptional vertex, we can classify its configurations (hence sms's) into two types.  A configuration $\config$ of $\integer A_{em}/\langle \tau^e\rangle$ (with $m>1$) is said to be of ``\emph{bottom-type}" (resp. ``\emph{top-type}") if the resulting configuration after tree pruning is $\config_{h,m}^-$ (resp. $\config_{\ell,m}^+$).  The two types distinguish a configuration from which $\tau$-orbit it lies in.  We denote $\smss_-(B_{e,m}^G)$ (resp. $\smss_+(B_{e,m}^G)$) for the set of sms's such that their corresponding configurations can be truncated to $\config_{h,m}^-$ (resp. $\config_{h,m}^+$) for some $h\in\{1,\ldots,e\}$.

\begin{corollary}\label{cor-disjoint-sms-type}
For the Brauer tree algebra $B_{e,m}^G$ with non-trivial multiplicity $m>1$, a simple-minded system of $B_{e,m}^G$ lies in either $\smss_-(B_{e,m}^G)$ or $\smss_+(B_{e,m}^G)$.  Moreover, there is a one-to-one correspondence between $\smss_-(A)$ and $\smss_+(A)$ given by Heller translate $\Omega$.
\end{corollary}
\begin{proof}
It suffice to prove for $G=\star$ as sms's are stably invariant and bottom/top parity depends solely on the position of modules in the stable AR-quiver.  The first statement is clear from previous corollary.  For the second statement, note that $\Omega$ induces an automorphism on the stable AR-quiver ${}_s\Gamma = \integer A_{em}/\langle \tau^e\rangle$ by sending $(i,j)$ to $(\overline{i-j+1},em+1-j)$.  Apply tree pruning to a configuration $\config$ of ${}_s\Gamma$ along with $\Omega\config$ by cutting the same leaf at each stage, one can see $\Omega$ swaps the type parity of the configuration.
\end{proof}

For $m=1$ case, tree pruning will not give us a well-defined type as every vertex of the Brauer tree are (non-)exceptional.  This undermines the non-bijection nature of the map $\frakf$ in Theorem \ref{thm-main} in the case of $A_n^n$.  However, comparing tree pruning on $B_{e,m}^\star$ with $m>1$ and the corresponding procedure on $B_{e,1}^\star$ gives us the following relation between the configurations of their stable AR-quiver.

\begin{prop}\label{prop-surject-config}
Let $\config$ be a configuration of $\integer A_{em}/\langle\tau^e\rangle$, for each $(x,y)\in\config$, we have $(x,y)=(x,\tilde{y})$ or $(x,e(m-1)+\tilde{y})$ for some $\tilde{y}\in\{1,\ldots,e\}$, then $\widetilde{\config}:=\{(x,\tilde{y})|(x,y)\in\config\}$ is a configuration of $\integer A_e/\langle \tau^e\rangle$.  In particular, the assignment induce a surjection from $\smss(A_e^{em})$ onto $\smss(A_e^e)$.
\end{prop}

\section{Mutation theories for self-injective Nakyama algebras}\label{sec:2tilt-mutations}
\subsection{Definitions of mutations on silting complexes and sms's}
We work with the ``Nakayama-stable" version of silting mutations introduced in \cite[Sec 5]{CKL}, which mutate tilting complexes to tilting complexes.

Let $T=T_1\oplus\cdots\oplus T_r$ be a complex in $K^b(\mathrm{proj}A)$.
If $X$ is a Nakayama-stable summand of $T$ such that for any Nakayama-stable
summand $Y$ of $X$, we have $Y=X$, then we call $X$ a \emph{minimal Nakayama-stable summand}.

\begin{definition}[\cite{AI,CKL}]\label{defn-silting-mutation}
For a (basic) tilting complex $T=T_1\oplus\cdots \oplus T_r$ (so each of $X_i$'s is indecomposable and they are pairwise non-isomorphic) and $X$ a Nakayama-stable summand of $T$, we write $T=X\oplus M$.
A \emph{left tilting mutation} of $T$
with respect to $X$, denoted by $\mu_X^-(T) = U_1\oplus\cdots \oplus U_r$,
is the complex with indecomposable summands $U_i$ are given as follows:
\begin{enumerate}
\item $U_i = T_i$ if $T_i$ is not a direct summand of $X$;
\item Otherwise, $U_i$ is the unique object appearing in the distinguished triangle:
\[
M' \to T_i \to U_i \to M'[1]
\]
where the first map is a minimal left $\add M$-approximation of $X_i$.
\end{enumerate}
The \emph{right tilting mutation} $\mu_X ^+(T)$ is defined similarly.
A tilting mutation with respect to $X$ is called \emph{irreducible} if $X$ is minimal Nakayama-stable.
\end{definition}

Analogous to mutation of tilting complexes, the mutation of sms's is given below.
\begin{definition}[Def 4.3 in \cite{D3})] \label{mutation-sms}  ,
Let $\mathcal{S} = \{X_1,\ldots, X_r\}$ be an sms of $\stmod{A}$ with $A$ finite dimensional self-injective,
and suppose $\mathcal{X}\subseteq \mathcal{S}$ is Nakayama-stable.
The \emph{left sms mutation} of $\mathcal{S}$ with respect to $\mathcal{X}$ is the set
  $\mu_{\mathcal{X}}^-(\mathcal{S}) = \{Y_1,\ldots,Y_r\}$ such that
\begin{enumerate}
\item $Y_j=\Omega^{-1}X_j$, if $X_j\in \mathcal{X}$
\item Otherwise, $Y_j$ is defined by the following distinguished triangle
$$\Omega X_j\rightarrow X \rightarrow Y_j \to X_j $$ 
where the first map is a minimal left $\mathcal{F(X)}$-approximation
of $\Omega X_j$.
\end{enumerate}
The \emph{right sms mutation} $\mu_{\mathcal{X}}^+(\mathcal{S})$ of $\mathcal{S}$ is defined similarly.
If we mutate a sms with respect to a minimal Nakayama-stable subset, then we call the mutation \emph{irreducible}.
\end{definition}
\begin{remark}
In our setting, the set of configurations of $\integer A_n$ inherits a mutation theory from that of $\smss(A_n^n)$, due to the bijection shown in \cite{CKL}.  This mutation is, however, different from the one used in \cite{S}.
\end{remark}
\begin{remark}\label{rmk-covering-theory}
Note that the Nakayama functor of $A_n^\ell$ sends $(x,y)$ to $(\overline{x+e},y)$ with $e=\gcd(n,\ell)$, and configurations of $\integer A_\ell/\langle \tau^n\rangle$ are $\tau^{e\integer}$-stable.  Therefore, the effect of performing an irreducible mutation on a sms (configuration) of $A_n^\ell$ can be observed on a corresponding irreducible mutation on $A_e^\ell$.  Hence, ``covering theory is compatible with mutation".
\end{remark}

\subsection{Combinatorial description of two-term tilting complexes.}
Given a tilting complex $T$, let $E_T$ denote the derived equivalent algebra $\End_{\catT}(T)$, and $F_T:D^b(\rmod{A})\to D^b(\rmod{E_T})$ be the associated derived equivalence.  By a result of Rickard \cite{Ric2}, any algebra derived equivalent to a Brauer tree algebra is also a Brauer tree algebra, hence, we sometimes call a tilting complex over $B_{e,m}^\star$ to be ``star-to-tree tilting complex" if $A$ is the Brauer star algebra $B_{e,m}^\star$.  However, there are infinitely many tilting complexes (even up to shifts and homotopy equivalences), and we should restrict to a much more refined subclass when studying the homological theories around these algebras.  Our choice in the current article is the set of two-term tilting complexes.  The main reason comes from the fact that every derived equivalence between representation-finite self-injective algebras given by a tilting complex is a composition of derived equivalences given by two-term tilting complexes shown by Abe and Hoshino \cite{AH}.  In \cite{AIR}, it is shown that the set of two-term tilting complexes of symmetric algebra is in bijection with the set of functorially finite torsion classes of its module category, emphasising the importance of two-term tilting complexes in the study of homological behaviour of a symmetric algebra.

We will use the combinatorial description of two-term tilting complexes from a mixture of results from \cite{SZ,RS, Adachi}.  In \cite{Adachi}, combinatorial descriptions are given to the so-called support $\tau$-tilting modules; since this is not our main interest, we will not go through the definitions of $\tau$-tilting theory, instead we just recall the result from \cite{AIR}, which says that the set of two-term silting complexes over a finite dimensional algebra $A$ is in order-preserving correspondence to the set of support $\tau$-tilting $A$-module, so that we can translate the results from \cite{Adachi} for our needs.  As we have mentioned, a silting complex is tilting if and only if it is Nakayama-stable; this translates into the following result in $\tau$-tilting theory:
For a finite dimensional algebra $A$, there is a mutation preserving correspondence between $\twotilt(A)$ and the set of Nakayama-stable support $\tau$-tilting $A$-modules.  This is also implicit from work of Mizuno \cite{Mizuno}.

For a self-injective Nakayama algebra $A_n^\ell$, the result of \cite{Adachi} gives us a combinatorial description of two-term tilting complexes over $A_n^\ell$ via triangulations on a punctured regular convex $n$-gon.

\begin{definition}[c.f. \cite{Adachi}]\label{defn-triangulation}
Let $i,j\in\{1,\ldots,n\}$, and $\mathcal{G}_n$ be a punctured regular convex $n$-gon (\emph{punctured $n$-disc}) with vertices labelled by $\{1,\ldots,n\}$ with counter-clockwise ordering.
\begin{itemize}
\item[(1)] An \emph{inner arc} $\langle j,i \rangle$ in $\mathcal{G}_n$ is a path from the vertex $i$ to the vertex $j$ 
homotopic to the boundary path $i,\overline{i+1},\cdots,\overline{i+l}=j$ such that $1<l \le n$. 
Then we call $i$ (respectively, $j$) a \emph{initial} (respectively, \emph{terminal}) point and
$\ell(\langle i,j\rangle):=l$ the \emph{length} of the inner arc.
\item[(2)] A \emph{projective arc} $\langle \bullet,j\rangle$ in $\mathcal{G}_{n}$ is a path from the puncture to the vertex $j$.
Then we call $j$ a {\it terminal} point.
\item[(3)] An \emph{admissible arc} is an inner arc or a projective arc.  We denote by $\mathcal{A}(n)$ the set of admissible arcs in $\mathcal{G}_n$. 
\item[(3)] Two admissible arcs in $\mathcal{G}_{n}$ are called \emph{compatible} if they do not intersect in $\mathcal{G}_{n}$ (except at their initial and terminal points).
\item[(4)] A \emph{triangulation} of $\mathcal{G}_{n}$ is a maximal set of distinct pairwise compatible admissible arcs.
We denote by $\mathcal{T}(n)$ the set of triangulations of $\mathcal{G}_{n}$, and by $\mathcal{T}(n;l)$ the subset of $\mathcal{T}(n)$ consisting of triangulations such that the length of every inner arc has length at most $l\leq n$.
\end{itemize}
\end{definition}
\begin{remark}
The original notation used by Adachi is $\langle i,j\rangle$ instead of $\langle j,i\rangle$.  This is due to the different vertex labelling and direction of composition of arrows on the quiver we use, so that the new notation still matches up with the terms appearing in two-term tilting complexes, as we will see in the following theorem.
\end{remark}

We also note that $\mathcal{T}(n)$ admits a mutation theory, namely, for a given triangulation $X\in\mathcal{T}(n)$ and an admissible arc $a\in X$, the (irreducible) mutation of $X$ with respect to $a$ is a unique triangulation $\mu_a(X)\in\mathcal{T}(n)$\index{mutation!of triangulation}.  This gives a partial order\index{partial order!on triangulations} structure on $\mathcal{T}(n)$ with the triangulation $\{\langle \bullet,i\rangle|i=1,\ldots,n\}$ being the unique maximal one.
Also recall from previous section that the set of tilting complexes (up to shifts and homotopy equivalences) admits a partial ordering given by $T\geq U$ if and only if $\Hom(T,U[>0])=0$, which is compatible with its mutation theory (c.f. \cite[5.11]{CKL}).
We restate the theorem of Adachi using two-term silting complexes instead of support $\tau$-tilting modules.
  
\begin{theorem}[\cite{Adachi}]\label{thm-adachi}
Let $n,\ell\in\mathbb{N}$,
\begin{enumerate}
\item The map $\langle j,i\rangle \mapsto (P_{\overline{j-1}}\to P_i)$ and $\langle \bullet,i\rangle\mapsto P_i$ the stalk complex concentrated in degree 0 induces a $\phi_-$ from $\mathcal{T}(n;\min\{\ell,n\})$ to the subset $\twosilt_-(A_n^\ell)$ of two-term silting complexes of $A_n^\ell$, which is order-preserving when $n\leq\ell$, i.e. $\phi_-(\mu_a(X))=\mu_{\phi_-(a)}^-(\phi_-(X))$ if $X\geq \mu_a(X)$.

Dually, the map $\langle j,i\rangle \mapsto (P_j\to P_{\overline{i+1}})$ and $\langle \bullet,i\rangle\mapsto P_i$ the stalk complex concentrated in degree $-1$ induces a map $\phi_+$ from $\mathcal{T}(n;\min\{\ell,n\})$ to the subset $\twosilt_+(A_n^\ell)$ of two-term silting complexes of $A_n^\ell$, which is anti-order-preserving when $n\leq\ell$, i.e. $\phi_+(\mu_a(X))=\mu_{\phi_+(a)}^+(\phi_+(X))$ if $X\geq \mu_a(X)$.

In particular, there is a bijection between $\twosilt_-(A_n^\ell)$ and $\twosilt_+(A_n^\ell)$.

\item $\twosilt(A_n^\ell) = \twosilt_-(A_n^\ell)\sqcup\twosilt_+(A_n^\ell)$.
\item For a two-term silting complex $T$ of $A_n^\ell$, $T\in\twosilt_-(A_n^\ell)$ (resp. $T\in\twosilt_+(A_n^\ell)$) if and only if all its indecomposable stalk complexes are concentrated in degree 0 (resp. $-1$).
\end{enumerate}
\end{theorem}

We now refine this result on tilting complexes.  Note that the Nakayama permutation of $A_n^\ell$ is a product of $e$ disjoint $n/e$-cycles, so the effect of applying Nakayama functor on a silting complex now manifests as turning the punctured $n$-disc anti-clockwise by $2\pi/(n/e)$.  Recall that two-term tilting complexes of self-injective algebras are just Nakayama-stable silting complexes, so they correspond to triangulations of punctured $n$-disc with a $2\pi/(n/e)$-rotation symmetry.  Such type of triangulations will then be in correspondence with triangulations on a punctured $e$-disc by identifying the punctured point of $n$-disc with punctured point of $e$-disc, and vertex $i$ with $ke+i$ for all $k=1,\ldots,n/e-1$ and $i=1,\ldots,e$.  

\begin{example}
Let $n=12, \ell=16$, Figure \ref{fig-triang} shows a triangulation of a punctured $12$-disc on the left, which is $2\pi/3$-rotational symmetric.  This triangulation can be identified with a triangulation of $4$-disc shown on the right.
\end{example}
\begin{figure}[h!]
\centering
\includegraphics[width=3.8in]{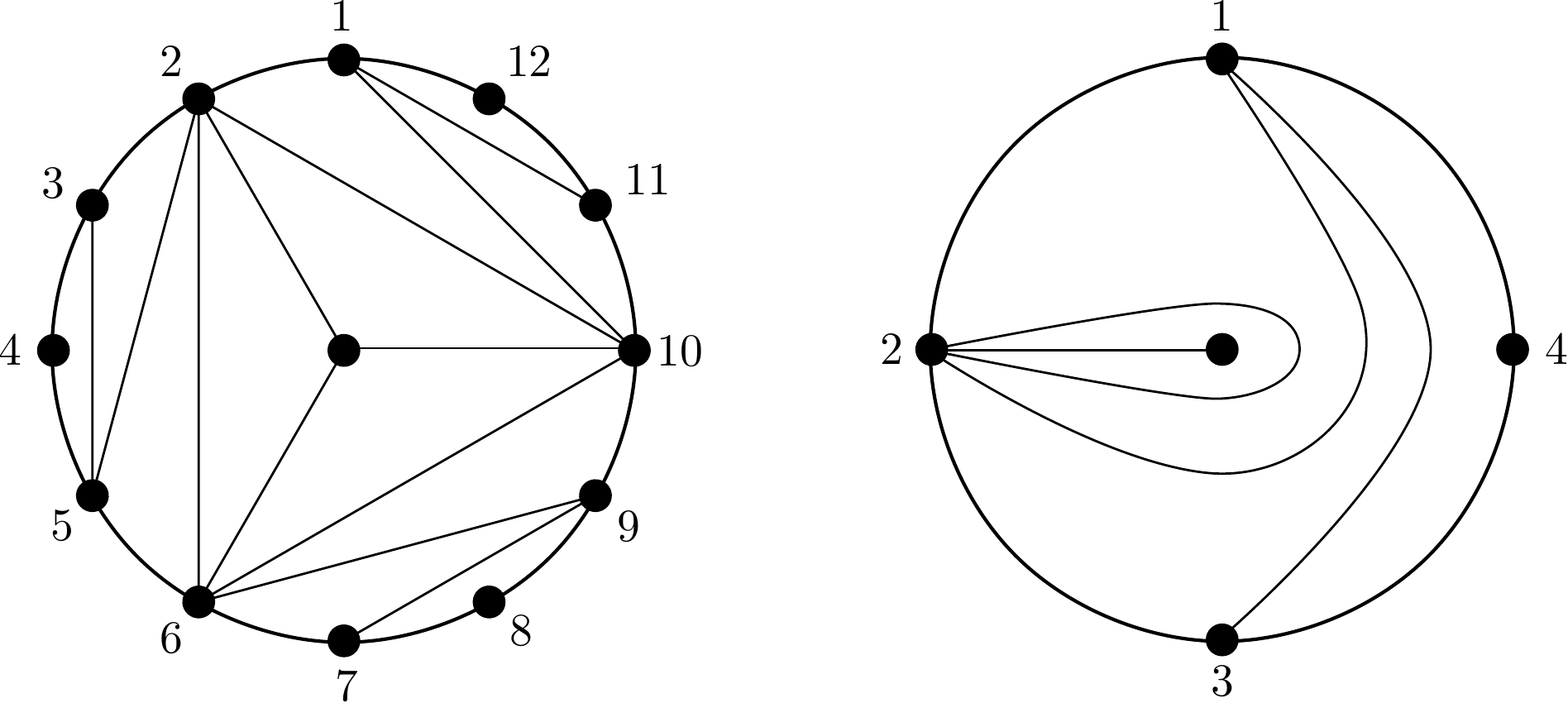}
\caption{Identifying rotational symmetric triangulation of $12$-disc and triangulation of $4$-disc}\label{fig-triang}
\end{figure}

Summarising, we have the following result:

\begin{theorem}\label{thm-SINakayama-2tilt}
For any $n,\ell\in\integer$ and $e=\gcd(n,\ell)$, there are bijections:
\begin{equation}\label{eqn-triang-tilt-biject}
\begin{array}{rcl}
\twotilt_-(A_n^\ell) \;\;\leftrightarrow  & \mathcal{T}(e) & \leftrightarrow\;\; \twotilt_-(A_e^\ell)\\
\twotilt_+(A_n^\ell) \;\;\leftrightarrow  & \mathcal{T}(e) & \leftrightarrow\;\; \twotilt_+(A_e^\ell)
\end{array}
\end{equation}
where those in top row are order-preserving and those in the bottom row are anti-order-preserving respectively.  In particular, we have mutation preserving bijections $\twotilt_\pm(A_n^\ell)\leftrightarrow \twotilt_\pm(A_e^\ell)$.
\end{theorem}
\begin{remark}
The reader should be careful when consider the case $n>\ell$.  Note that rotational symmetry restricts the lengths of admissible arcs to be less than $e\leq \ell$, hence the assignment from a rotational symmetric triangulation in $\mathcal{T}(n;\ell)$ to a triangulation in $\mathcal{T}(e)$ as before is still well-defined and (anti-)order-preserving.
\end{remark}

This result is a refinement of the covering theory for derived categories of representation-finite self-injective algebras used in \cite{Asa}.  Reader familiar with covering theory in \cite{Asa} would naturally expect such a result as a consequence of \cite[5.11]{CKL}.  We note that, it is not clear from the proofs of \cite{Asa} whether all (two-term) tilting complexes of $A_n^\ell$ can be obtained by using covering theory of the (two-term) tilting complexes of $A_e^\ell$; this result shows an affirmative answer.

\subsection{Constructing Brauer trees from a two-term tilting complex}
Given a two-term tilting complex $T$ of $B_{e,m}^\star = A_e^{em}$, there is a simple construction to determine the Brauer tree $G$ associated to $E_T$ using result of Schaps and Zakay-Illouz \cite{SZ,RS}, which we will go through in the next section.  Here, we use Schaps-Zakay-Illouz construction to obtain $G$ directly from the triangulation of a punctured $e$-disc.

Consider a triangulation $X\in\mathcal{T}(e)$, for each vertex $i$ on the punctured disc, we distinguish some sets of inner arcs in $X$ as follows
\begin{equation}\label{eqn-LRset}
\begin{array}{l}
A_i^-(X) = \{\langle j,i\rangle| j\in \{1,\ldots,e\}\},\\
A_i^+(X) = \{\langle i,k\rangle| k\in \{1,\ldots,e\}\}.
\end{array}
\end{equation}
Note that $X$ is the disjoint union of projective arcs and arcs in $A_i^-(X)$ (resp. $A_i^+(X)$) over all $i$.
One can now construct the Brauer tree $G$ associated to the endomorphism ring of $\phi_-(X)$ or $\phi_+(X)$ using result in \cite{SZ} as follows.

\begin{prop}\label{prop-triangualtion-to-BTree}
Let $X$ be a triangulation of a punctured disc, construct a pair of Brauer trees $G_X^-$ and $G_X^+$ as follows.
\begin{enumerate}[(1)]
\item Let $\{v_0,v_1,\ldots,v_e\}$ be vertices of $G_\pm$.  For each projective arc $\langle\bullet,i\rangle\in X$, connect $v_0$ and $v_i$ by an edge.

\item For each $i\in \{1,\ldots,e\}$ and each arc in $\langle j,i\rangle\in A_i^-(X)$ (resp. $\langle i,k\rangle \in A_i^+(X)$), connect the vertices $v_i$ and $v_{\overline{j-1}}$ of $G_X^-$ (resp. $v_i$ and $v_{\overline{k+1}}$ of $G_X^+$) by an edge.
\end{enumerate}
Then $G_X^-$ (resp. $G_X^+$) with exceptional vertex $v_0$ and multiplicity $m$ is the precisely the Brauer tree $G$ such that $B_{e,m}^G\isom \End_{K^b(\proj{B_{e,m}^\star})}(\phi_-(X))$ (resp. $\End_{K^b(\proj{B_{e,m}^\star})}(\phi_+(X))$).
\end{prop}
\begin{proof}
We prove for the minus part of the proposition; the plus part can be done dually.

Let $\{v_0,\ldots,v_n\}$ be the set of vertices of the Brauer tree $G$, then for any projective arc $\langle\bullet,j\rangle\in X$, we put an edge connecting $v_0$ and $v_j$, then for any inner arc $\langle i,k\rangle$, we put an edge connecting $v_{i-1}$ and $v_k$.  Therefore, for each arc $a\in A_i^-(X)$, $\phi_-(a)$ has degree 0 component $P_i$, and any other arc $a$ attached to $i$ not in $A_i^-(X)$ will send to a pretilting complex with degree $-1$ component $P_{i-1}$ under $\phi_-$.  According to Theorem 3 of \cite{SZ}, the counter-clockwise ordering of edges around each vertices can then be chosen to be compatible with the cyclic ordering on $\{1,\ldots,e\}$, i.e. $E_{i_1},\ldots,E_{i_r}$ is the counter-clockwise ordering of edges around $v_k$, connected with $v_{i_1},\ldots,v_{i_r}$ respectively, if and only if $i_1<i_2<\cdots<i_r$ in $\{1,\ldots,e\}$.  By the main theorem of \cite{SZ}, the tree constructed this way is then the Brauer tree $G$ with exceptional vertex $v_0$ of multiplicity $m$.
\end{proof}

Let $\BrTree(e,m)$ be the set of Brauer tree with $e$ edges and multiplicity $m$.  This proposition says that we obtain a pair of well-defined map $\psi_\pm:\mathcal{T}(e)\to \BrTree(e,m)$ given by $X\mapsto G_X^\pm$ for any $m>1$.

\begin{figure}[hbtp!]
\centering
\includegraphics[width=5in]{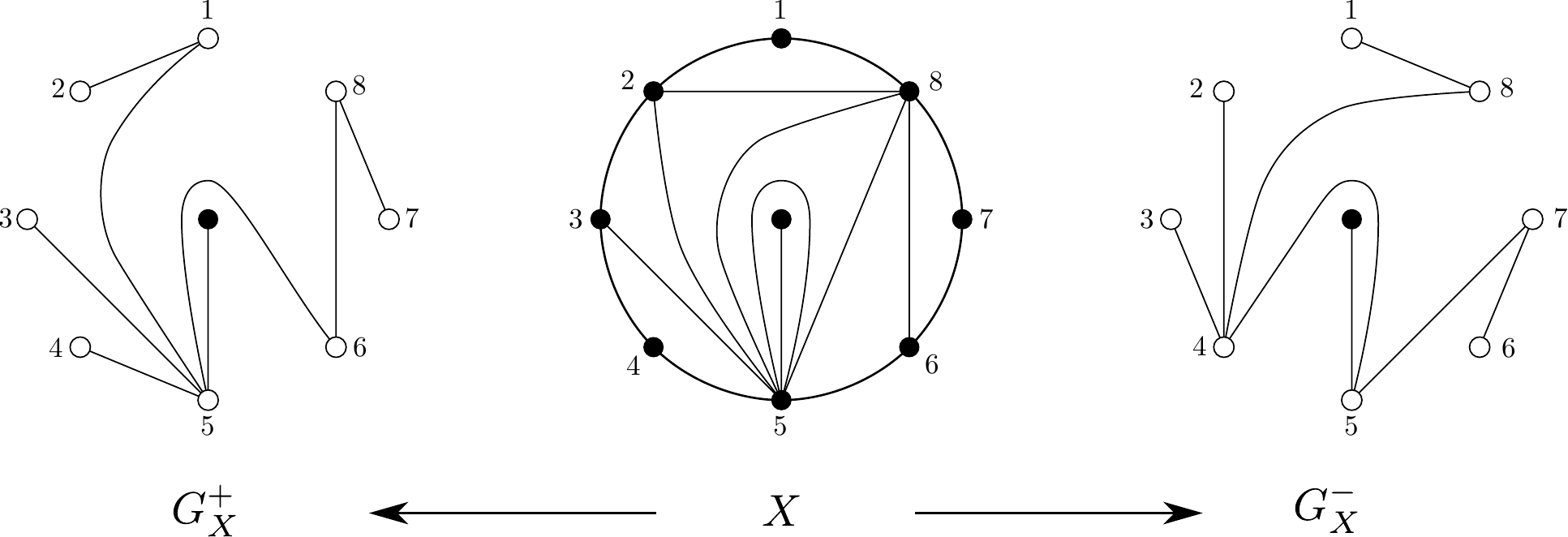}
\caption{Brauer trees from a triangulation of punctured disc}\label{fig-triangulation-to-BTree}
\end{figure}

\begin{corollary}\label{cor-edge-min-naka-stable-summand-corresp}
Let $T$ be two-term tilting complexes in $\twotilt_-(A_n^\ell)$, and $G$ be $\psi_-\phi_-^{-1}(T)$, then each minimal Nakayama-stable summand $M$ of $T$ corresponds to an edge of $G$.  Moreover, under this correspondence, each minimal Nakayama-stable summands concentrated in degree 0 corresponds to an edge emanating from the exceptional vertex of $G$.
\end{corollary}
\begin{proof}
Immediate from combining Proposition \ref{prop-triangualtion-to-BTree} with Theorem \ref{thm-adachi} and Theorem \ref{thm-SINakayama-2tilt}.
\end{proof}
\begin{remark}
Analogous statement holds for any two-term tilting complex in $\twotilt_+(A_n^\ell)$.
\end{remark}

\subsection{Some properties of mutations on two-term tilting complexes}  
Mutation of tilting complex can be reformulated as mutation on the class of derived equivalent algebras:  for a summand $P$ of an algebra $A$, let $T$ be the tilting mutation $\mu_X^-(A)$, then we can define the left algebra mutation as the algebra $E_T$.  On the class of Brauer tree algebras, this gives a mutation on the Brauer trees.  This mutation has been given in several literature already \cite{KZ,Kauer,Aiha}.  We recommend \cite{Aiha} for the most concise and precise description that is sufficient for our needs.

\begin{definition}[Mutation of Brauer tree]\label{defn-muBTree}
Let $G$ be a Brauer tree, and $i$ be an edge of $G$, the \emph{left mutation}\index{Brauer tree!mutation}\index{mutation!of Brauer tree} of $G$ at $i$, denoted $\mu_i^-(G)$, can be constructed as follows.  Suppose the vertices attached to $i$ are $u$ and $v$, with $j$ and $k$ being the previous edges in the cyclic ordering around $u$ and $v$ respectively.   The mutated tree is given by removing the edge $i$ from $G$, and replace with an edge $i'$ connected to $j$ and $k$.  In particular, if (without loss of generality) $u$ is only of valency one (i.e. an extremal vertex), then $i'$ is attached to $u$ again.

Similarly, define the right mutation $\mu_i^+(G)$ by removing $i$ and connecting $i'$ to the \emph{next} edges in the cyclic ordering around $u$ and $v$.  The two mutations can be visualised as in Figure \ref{fig-btreeMutate}.
\end{definition}
\begin{figure}[hbtp!]
\centering
\includegraphics[width=5in]{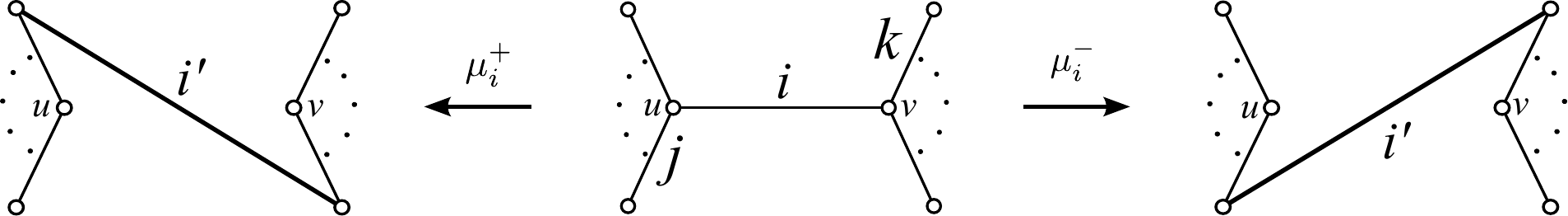}
\caption{Mutation of Brauer tree.}\label{fig-btreeMutate}
\end{figure}

Needless to say, mutation of Brauer trees is compatible with mutations of (two-term) tilting complexes and triangulations on punctured disc.  In this subsection, we will only consider the case when the algebra is $B_{e,m}^\star$, so if $T$ is a star-to-tree tilting complex which takes the Brauer star to a Brauer tree $G$, and let $i$ be an edge correspond to a summand $X$ of $T$, then $\mu_X^\pm(T)$ is a star-to-tree tilting complex which takes the Brauer star to the mutated Brauer tree $\mu_i^\pm(G)$.

From \cite[Theorem D]{CKL} (or \cite[Thm 3.5]{Aiha2}), we know that every tilting complex of $A_n^\ell$ can be obtained by a sequence of irreducible left (or right) mutations starting from $A_n^\ell$.  Our aim now is to find some ``canonical sequence" to obtain any give two-term tilting complex.  We start with some sufficient criteria for a mutated tilting complex to be two-term.

\begin{lemma}\label{lem-2tilt-trick}
Let $T$ be a tilting complex concentrated in non-positive (resp. non-negative) homological degrees, and $X$ be a minimal Nakayama-stable summand of $T$, we have the following:
\begin{enumerate}[(1)]
\item If $\mu_X^-(T)$ (resp. $\mu_X^+(T)$) is two-term, then so is $T$.
\item If $T$ is two-term and $X$ is a direct sum of stalk complexes concentrated in homological degree 0, then $\mu_X^-(T)$ (resp. $\mu_X^+(T)$) is two-term concentrated in homological degree $-1$ (resp. $+1$) and $0$.
\end{enumerate}
\end{lemma}
\begin{proof}
We prove the minus-version of the statements, plus-version can be done analogously.

(1):  Recall from \cite{AI} that there is a partial order on the tilting complexes defined by $T\geq U$ if $\Hom_\catT(T,U[i])=0$ for all $i> 0$.  As $T$ and $\mu_X^-(T)$ are both concentrated in non-positive degrees, it follows from \cite[2.9]{Aiha2} that, $A\geq \mu_X^-(T)\geq A[1]$ and $A\geq T \geq A[l]$ for some $l\geq 1$.  We also have $T \gneqq \mu_X^-(T)$ from  \cite[2.35]{AI} (for self-injective version see \cite[5.11]{CKL}).  These combine to give $A \geq T \gneqq \mu_X^-(T)\geq A[1]$, and so $l=1$, which means $T$ is two-term by \cite[2.9]{Aiha2}.

(2):  This is easy to see from the definition of mutation, for $T=X\oplus M$ with $\mu_X^-(T)=Y\oplus M$, then $Y$ is the cone of a morphism from stalk complex concentrated in degree 0 to a two-term complex concentrated in degree 0 and $-1$, so $Y$ is two-term as well.
\end{proof}

\begin{prop}\label{thm-2tiltAlgo}
Suppose $T\in\twotilt_-(A_n^\ell)$ (resp. $T\in\twotilt_+(A_n^\ell)$), then $T$ can be obtained by $h$ irreducible left (resp. right) mutations starting from $A_n^\ell$ (resp. $A_n^\ell[1]$) for some $h<e=\gcd(n,\ell)$.
\end{prop}
\begin{proof}
Again, we only prove the minus-version of the statement.

Let $U$ be the unique two-term tilting complex in $\twotilt_-(A_e^{em})$ with any $m>1$ corresponding to $T$ under the correspondence of Theorem \ref{thm-SINakayama-2tilt}.  Then $E_T\isom B_{e,m}^G$ with $G=\psi_-\phi_-^{-1}(T)$ with valency of exceptional vertex being $h<e$.

By Lemma \ref{lem-2tilt-trick} and Corollary \ref{cor-edge-min-naka-stable-summand-corresp}, it suffices to prove the following combinatorial result:  The Brauer tree $G$ with valency $e-h$ at the exceptional vertex can be obtained by $h$ left mutations at edges attached to the exceptional vertex.  Thanks to \cite[2.33]{AI}, this can be proved by finding an algorithm to obtain the Brauer star from the Brauer tree using right mutations, such that after each mutation, the valency of exceptional vertex is increased by 1, which can be found in the proof of \cite[3.1]{KZ}.
\end{proof}
\begin{example}
Let $T=\bigoplus_{i=1}^6T_i\in\twotilt(B_{e,m}^\star)$ be given by 
\[\begin{array}{rccrccrcc}
T_1 = (0   & \to & P_2), & T_2 = (P_3 & \to & P_2), & T_3 = (\;\;\,  0 & \to & P_4),\\
T_4 = (P_1 & \to & P_4),  & T_5 = (P_1 & \to & P_5), & T_6 = (P_1 & \to & P_6).
\end{array}\]
Use Proposition \ref{prop-triangualtion-to-BTree} to obtain a Brauer tree and apply the proof of Proposition \ref{thm-2tiltAlgo} to obtain the mutation sequence $T=\mu_{P_5}^-\mu_{P_6}^-\mu_{P_1}^-\mu_{P_3}^-(A)$.  The details of this computation is shown in Figure \ref{fig-2tiltAlgo}.
\begin{figure}[hbtp!]
\centering
\includegraphics[width=3in]{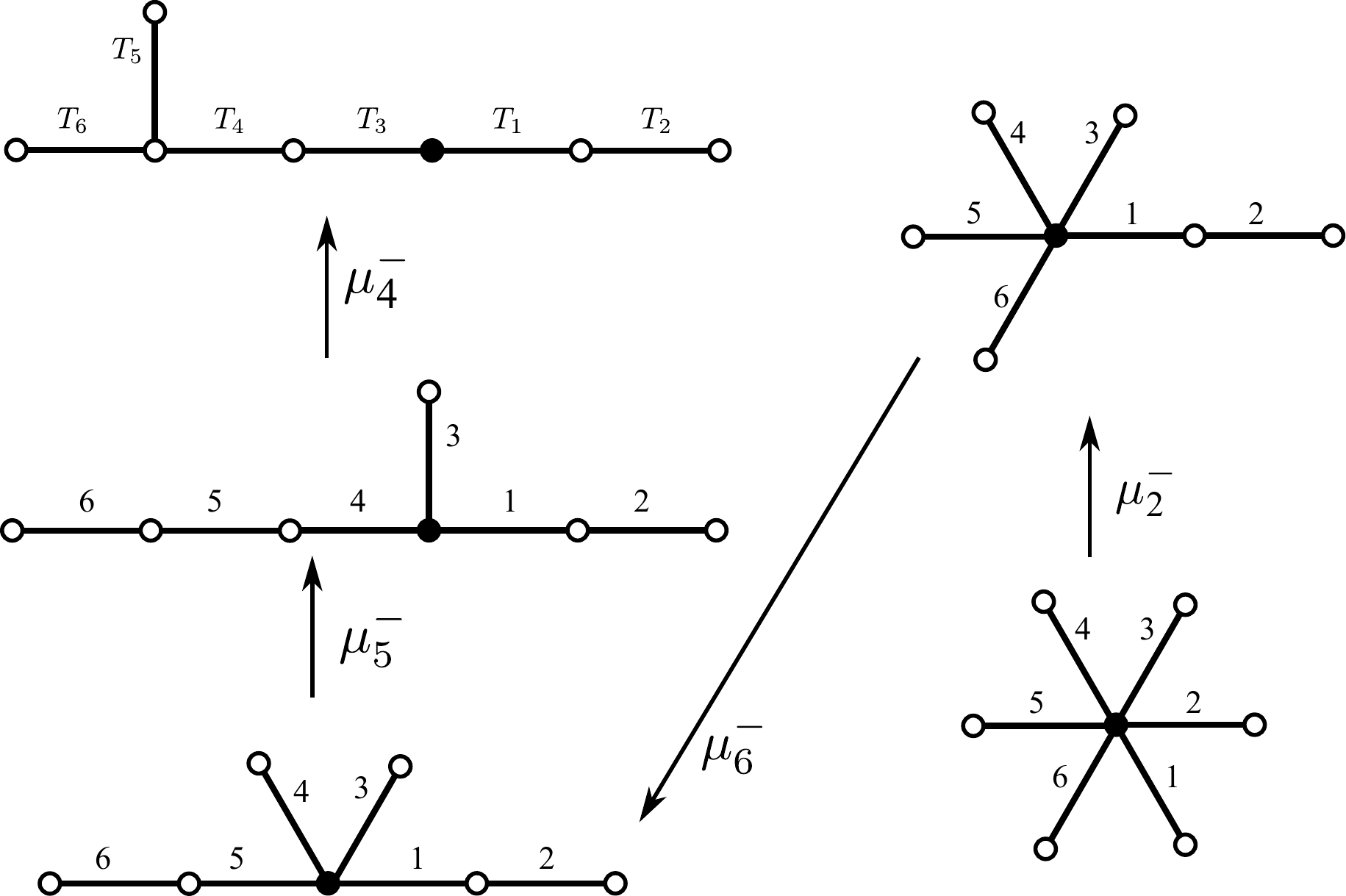}
\caption{An example for Proposition \ref{thm-2tiltAlgo}}\label{fig-2tiltAlgo}
\end{figure}
\end{example}

\subsection{Observations on mutation of simple-minded systems}
Having known how to obtain the sequence of mutation to reach any given two-term tilting complex, we need some observations on the effect of mutation on simple-minded system.  We will only do the case $A_e^{em}=B_{e,m}^\star$, analogous result can be obtained by covering theory (c.f. Remark \ref{rmk-covering-theory}).

\begin{lemma}\label{lem-mutateSimples}
Let $\sms$ be the set of simple $B_{e,m}^G$-modules, and $S_i$ be a simple $B_{e,m}^G$-module corresponding to an edge $i$ of $G$.  Then an irreducible left mutation $\mu_{S_i}^-(\sms)$ replaces exactly two (indecomposable) modules in $\sms$ if $i$ is a leaf, or replaces exactly three modules in $\sms$ otherwise.
In particular, at most three (indecomposable) modules in any simple-minded system of any Brauer tree algebras will be replaced after performing an irreducible left mutation.
\end{lemma}
\begin{proof}
The first statement follows from straightforward calculation using definition of mutation, or alternatively, implicitly implied by a result of Okuyama in his unpublished preprint \cite[Lemma 2.1]{Ok}, which also appears in the proof in \cite[Lemma 3.4]{Aiha}.

Now suppose $\sms$ is an arbitrary simple-minded system of arbitrary Brauer tree algebra, we know that there is a stable equivalence $\phi$ making $\sms$ a simple-image.  The last statement now follows from the fact that $\mu_{X}^-(\sms) = \phi^{-1}(\mu_{\phi X}^-(\phi\sms))$.
\end{proof}
\begin{remark}
Dually, the same result holds for irreducible right mutation.  We also remark that, in the notation of Figure \ref{fig-btreeMutate}, then the modules that are replaced after mutating at $S_i$ is precisely the simple modules $S_j$ and/or $S_k$.
\end{remark}

Recall from Proposition \ref{prop-surject-config} that for each $\sms\in\smss(B_{e,m}^G)$ with $m>1$, there is a sms $\widetilde{\sms}\in\smss(B_{e,1}^G)$; the following observation connect the mutation action between them.

\begin{lemma}\label{lem-mutateSimples-multi-free}
Suppose $\sms$ is a simple-minded system of $B_{e,m}^G$ for some Brauer tree $G$ with $m>1$.  Then for any $X\in\sms$, $\widetilde{\mu_X^\pm(\sms)} = \mu_{\widetilde{X}}^\pm(\widetilde{\sms})$.
\end{lemma}
\begin{proof}
Suppose $\sms$ is image of simple $B_{e,m}^H$-modules, then $\widetilde{\sms}$ is image of simple $B_{e,1}^H$-module.  Mutating at $X\in\sms$ corresponding to an edge $i$ in $H$ implies that we have $\mu_X^\pm(\sms)$ as image of simple $B_{e,m}^{\mu_i^\pm(H)}$-modules.  So $\widetilde{\mu_X^\pm(\sms)}$ is the image of simple $B_{e,1}^{\mu_i^\pm(H)}$-modules given by $\mu_{\widetilde{X}}^\pm(\widetilde{\sms})$.
\end{proof}

The following observation is crucial for the proof of the main theorem.  Also recall from previous section that the type of a configuration indicates the rim where the simple module lies after tree pruning (c.f. Lemma \ref{cor-cutTree}).

\begin{prop}\label{prop-respectTauOrbit}
Let $\sms$ is a simple-minded system of $B_{e,m}^\star$ which is a $B_{e,m}^G$-simple-image for some Brauer tree $G$, with $m>1$ and the valency of exceptional vertex of $G$ being $h>1$.  If $X\in\sms$ is the image of a simple $B_{e,m}^G$-module corresponding to an edge attached to the exceptional vertex, then $\sms$, $\mu_X^+(\sms)$, and $\mu_X^-(\sms)$ are all of the same type.
\end{prop}
\begin{proof}
We will use the labelling for edges as in Definition \ref{defn-muBTree}
and suppose $X$ corresponds to an edge $i$ attached to the exceptional vertex $v$.  
We show the proof for left mutation; right mutation is done analogously.  
It follows from Lemma \ref{lem-mutateSimples} that the effect of irreducible mutation on the configuration is to replace at most three of the vertices.  
After tree pruning, the effect on the configuration $\config_h$ corresponding to $\sms$ is either $\config_{h,m}^-$ (bottom-type) or $\config_{h,m}^+$ (top-type).  
After the mutation at $X$, all but one module corresponding to the edges attached to $v$ remains unchanged in $\mu_X^-(\sms)$.  So after pruning the mutated tree, we are left with a configuration $\config_{h-1}$ of $\integer A_{(h-1)m}$ such that the $h-2$ vertices in $\config_{h-1}$ lie in the same rim as their corresponding vertices in $\omega_e^{m}(\config_{h-1})\subset\config_h$.  When $h>2$, this forces $\config_{h-1}$ and $\config_h$ to be of the same type, and hence $\sms$ and $\mu_X^-(\sms)$.  For $h=2$, apply tree pruning on $G$ until reaching a Brauer star with 2 edges.  Now consider an irreducible (left or right) mutation of the set of simple $B_{2,m}^\star$-modules $\{(1,1),(2,1)\}$, one will obtain either $\{(2,2m),(1,1)\}$ or $\{(1,2m), (2,1)\}$.  Hence the three sms's are of the same type.
\end{proof}

\section{Proof of Theorem \ref{thm-main}}\label{sec:2tilt-proof}
This entire section is devoted to proving the main theorem \ref{thm-main}.
For convenience, we denote $A$ the Brauer star algebra $B_{e,m}^\star$ with multiplicity $m>1$ throughout this section.  Recall that for any algebra $\Lambda$, we denote by $\sms_\Lambda$ as the set of simple $\Lambda$-modules, and we will always identify sms's with configurations.  

Our plan to prove Theorem \ref{thm-main} is to first show that it holds (i.e. $\frakf$ is a bijection) for the case $A=B_{e,m}^\star$, then extends to $A_n^\ell$ with $\ell\neq\gcd(n,\ell)$.  Afterwards, we show $\frakf$ is surjective non-injective for other cases of $A_n^\ell$.

We will achieve the first goal by showing this:
\begin{theorem}\label{thm-main-btree}
Restricting $\frakf:\twotilt(A)\to\smss(A)$ to the disjoint subsets $\twotilt_\pm(A)$, we have bijections $\frakf_\pm:\twotilt_\pm(A)\to\smss_\pm(A)$.  Moreover, $\frakf_\pm$ preserve mutations, that is, for all $T\in\twotilt_-(A)$ and $T'\in\twotilt_+(A)$, with indecomposable pretilting summand $X$ and $X'$ respectively, such that $\mu_X^-(T)$ and $\mu_{X'}^+(T')$ are two-term tilting complexes, then 
\[
\frakf_-(\mu_X^-(T)) = \mu_{\widetilde{X}}^-(\frakf_-T) \;\;\mbox{ and }\;\; \frakf_+(\mu_{X'}^+(T')) = \mu_{\widetilde{X'}}^+(\frakf_+T')
\]
for some indecomposable $A$-modules $\widetilde{X}$ and $\widetilde{X'}$.
\end{theorem}
\begin{proof}
Mutation preserving property is inherited from the composition of mutation preserving maps $\twotilt(A)\to\tilt(A)$ and $\tilt(A)\to\smss(A)$.  This is by combining the following lemmas \ref{lem-frakf-well-def}, \ref{lem-frakf-surj}, \ref{lem-frakf-inj}. 
\end{proof}

\begin{lemma}\label{lem-frakf-well-def}
$\frakf_\pm:\twotilt_\pm(A)\to \smss_\pm(A)$ are well-defined.
\end{lemma}
\begin{proof}
For each $T\in\twotilt_-(A)$, by Proposition \ref{thm-2tiltAlgo}, $T$ can be obtained by iterative irreducible left mutation with respect to stalk complexes starting from $A$.  Since stalk complexes correspond to edges attached to the exceptional vertex in $G$ where $E_T\isom B_{e,m}^G$, and mutation is preserved when restricting a standard derived equivalence to stable equivalence, we can repeatedly apply Proposition \ref{prop-respectTauOrbit} and $\sms=\frakf(T)$ has the same type as $\frakf(A)=\sms_A\in\smss_-(A)$, so $\sms\in\smss_-(A)$.

If $T\in\twotilt_+(A)$, then $T$ can be obtained by iterative irreducible right mutation with respect to stalk complexes (which are concentrated in degree $-1$) starting from $A[1]$.  So $U=T[-1]$ can be obtained by iterative irreducible right mutation with respect to stalk complexes concentrated in degree 0.  Again, apply Proposition \ref{prop-respectTauOrbit} repeadtly, then $\sms = \stinv{F_U}(\sms_{E_U})$ has the same type as $\frakf(A)$.  Since $E_U\isom B_{e,m}^G\isom E_T$ for some Brauer tree $G$, so $\sms_{E_U}=\sms_{E_T}$.  So we have
\begin{eqnarray*}
\frakf(T) &=& \stinv{F_T}(\sms_{E_T}) \\
 & = & \stinv{F_{U[1]}}(\sms_{E_U})\\
 & = & \stinv{F_U\circ [-1]}(\sms_{E_U})\\
 & = & \Omega^{-1}\circ\stinv{F_U}(\sms_{E_U})\\
 & = & \Omega^{-1}\sms
\end{eqnarray*}
Note that the third equality follows from the fact that standard derived equivalence $F_T=F_{U[1]}$ is naturally isomorphic to the composition $F_U\circ [-1]$, and the fourth equality follows from the fact that the quotient functor $D^b(\rmod{A})\to\stmod{A}$ is triangulated and so the inverse suspension functor $[-1]$ restricts to $\Omega$.  As $\sms$ is of bottom-type, using Corollary \ref{cor-disjoint-sms-type}, $\frakf(T)=\Omega^{-1}(\sms)$ is of top-type.
\end{proof}

Recall the following result implicit from \cite{SZ,RS}.
\begin{prop}\label{prop-SZRS}
For each Brauer tree $G$ with $e$ edges and multiplicity $m>1$, there is a pair of two-term tilting complexes $T_\pm\in\twotilt_\pm(A)$ such that $E_{T_\pm}\isom B_{e,m}^G$.  In particular, $\psi_\pm\phi_\pm^{-1}:\twotilt_\pm(A)\to \BrTree(e,m)$ are sujective.
\end{prop}
\begin{proof}
As remarked in \cite[Example 1]{RS}, one can put extra combinatorial data on $G$ resulting in so called Brauer tree with completely folded pointing, then using main theorem of \cite{SZ} to construct a two-term tilting complex concentrated in degree 0 and 1 with stalk summands concentrated in degree 0.  So we can just shift this complex to obtain $T_+\in\twotilt(A)$.  To get $T_-$ one uses the dual pointing of completely folded pointing (see \cite[Example1]{RS}) and apply Schaps-Zakay-Illouz correspondence, but without shifting this time.
\end{proof}

\begin{lemma}\label{lem-permute-label}
Suppose $T_\pm\in\twotilt_\pm(A)$ with $G_\pm=\psi_\pm\phi_\pm^{-1}(T)$.  Every $T'_\pm\in\twotilt_\pm(A)$ with $\psi_\pm\phi_\pm^{-1}(T')=G_\pm$ is obtained by cyclically permuting the label of projective indecomposable modules in the components of $T_\pm$.  
\end{lemma}
\begin{proof}
The statement follows by observeing closely the construction and proof of main theorem of \cite{SZ}.  To be slightly more precisely, two-term tilting complex $T$ with $E_T\isom B_{e,m}^G$ corresponds to a completely folded pointing (or its dual) Brauer tree with a choice of non-exceptional vertex.  Changing this choice corresponds to cyclically permuting the label of projective indecomposable modules in the components of $T$.
\end{proof}
\begin{remark}
It follows that $\psi_\pm$ induces two different bijections between the set of triangulations of a punctured $e$-disc \emph{up to rotations} and the set Brauer trees with an exceptional vertex.
\end{remark}

\begin{lemma}\label{lem-frakf-surj}
If $\sms\in\smss(A)$ is in the image of $\frakf_-$ (or $\frakf_+$), then $\tau^n\sms$ for any $n\in\integer$ is also in the image of $\frakf_-$ (resp. $\frakf_+$).  In particular, $\frakf_\pm$ are surjective.
\end{lemma}
\begin{proof}
Again, we prove only for $\sms\in\smss_-(A)$, the other case is analogous.  The first statement follows from previous lemma \ref{lem-permute-label}, as changing labelling of projective modules in $T$ corresponds to changing the $x$-coordinate of the configuration corresponding to $\frakf(T)$.  

Since every sms of $A$ is simple-image, we have $\sms=\phi(\sms_{B_{e,m}^G})$ for some Brauer tree $G$ and stable equivalence $F:\stmod{B_{e,m}^G}\to\stmod{A}$.  By Proposition \ref{prop-SZRS}, we can find a $T'\in\twotilt_-(A)$ with $E_{T'}\isom B_{e,m}^G$.  
Let $\sms'=\frakf(T')$, so $\sms'$ can be obtained from $\sms$ by a stable auto-equivalence.  Recall from \cite{Asa2} that any stable auto-equivalence of self-injective Nakayama algebra is generated by the its Picard group and $\Omega$, hence $\sms'=\Omega^h\sms$ for some $h$.  Since $\sms'$ and $\sms$ are of the same type by Lemma \ref{lem-frakf-well-def}, $\sms'=\tau^h\sms$.  Surjectivity of $\frakf_-$ now follows.
\end{proof}

\begin{lemma}\label{lem-frakf-inj}
$\frakf_\pm:\twotilt_\pm(A)\to\smss_\pm(A)$ are injective.
\end{lemma}
\begin{proof}
We prove only for the minus version.  Suppose $T,T'\in\twotilt_-(A)$ with $\sms=\frakf(T)=\frakf(T')=\sms'$.  This implies $E_T\isom E_{T'}$, or equivalently $\psi_-\phi_-^{-1}(T)=G=\psi_-\phi_-^{-1}(T')$.  By Lemma \ref{lem-permute-label}, we have $T$ is given by permuting labels of components of $T'$.  As in previous Lemma \ref{lem-frakf-surj}, this implies $\sms=\tau^h\sms'$ for some $h$.  Hence $\sms'$ is $\tau^h$-stable.  But this means that the corresponding permutation on the labels on projective modules in $T'$ only permutes the summands of $T'$.  Hence $T=T'$.
\end{proof}

We now extends Theorem \ref{thm-main-btree} to some of the self-injective Nakayama algebras.
\begin{theorem}
For fixed $n,\ell\in\integer_{>0}$ with $\ell\neq e:=\gcd(n,\ell)$, restricting $\frakf:\twotilt(A_n^\ell)\to\smss(A_n^\ell)$ to the disjoint subsets $\twotilt_\pm(A_n^\ell)$, we have bijections $\frakf_\pm:\twotilt_\pm(A_n^\ell)\to\smss_\pm(A_n^\ell)$.  Moreover, $\frakf_\pm$ are mutation preserving.
\end{theorem}
\begin{proof}
Recall that there is a one-to-one correspondence between $\smss(A_n^\ell)$ and $\smss(A_e^\ell)$, as they can both be identified as $\tau^{e\integer}$-stable configurations of $\integer A_\ell$.  Note that $A_e^\ell$ is just Brauer star algebra $B_{e,m}^\star$ with $m=\ell/e>1$.  This correspondence respects mutations, in the sense that an irreducible mutation of sms in $\smss(A_e^\ell)$ corresponds to an irreducible (Nakayama-stable) mutation of sms in $\smss(A_e^\ell)$ (see Remark \ref{rmk-covering-theory}).  Also, the bijections $\twotilt_\pm(A_n^\ell)\leftrightarrow \twotilt_\pm(A_e^{em})$ are also mutation preserving (Theorem \ref{thm-SINakayama-2tilt}).   Therefore, we have composition of mutation preserving maps
\[
\twotilt_\pm(A_n^\ell)\to \twotilt_\pm(A_e^{em}) \to \smss_\pm(A_e^{em})\to \smss_\pm(A_n^\ell),
\]
and all the maps are bijective.  The theorem follows.
\end{proof}

Finally, we prove Theorem \ref{thm-main} for the remaining cases $A_{\ell k}^\ell$.  Note that putting put $e=\ell$ and $k=1$, the algebra $A_e^e$ is the multiplicity-free Brauer star algebra.
\begin{theorem}\label{thm-main-naka0}
For any fixed $e,\ell\in\integer_{>0}$ with $e=\ell k$ for some $k\neq 0$.  $\frakf:\twotilt(A_e^\ell)\to\smss(A_e^\ell)$ is surjective non-injective and preserves mutations.
\end{theorem}
\begin{proof}
We have
\begin{equation}\label{eqn-compose-mutation-resp-maps}
\twotilt(A_e^\ell)\leftrightarrow \twotilt(A_e^{em}) \leftrightarrow \smss(A_e^{em}) \twoheadrightarrow \smss(A_e^e) \leftrightarrow \smss(A_e^\ell)
\end{equation}
with all the left-to-right maps preserving mutations by Theorem \ref{thm-SINakayama-2tilt}, Theorem \ref{thm-main-btree}, Lemma \ref{lem-mutateSimples-multi-free}, Remark \ref{rmk-covering-theory}.  Take canonical stalk tilting complex $A_e^\ell$, it is easy to see that this maps to $\sms_{A_e^\ell}$ along the composition of maps in (\ref{eqn-compose-mutation-resp-maps}).  Since all self-injective Nakayama algebras are (strongly left) tilting-connected and the composition (\ref{eqn-compose-mutation-resp-maps}) respects mutation, this implies that the composition (\ref{eqn-compose-mutation-resp-maps}) is precisely $\frakf$.  Now it remains to show that $\smss(A_e^{em})\twoheadrightarrow\smss(A_e^e)$ is not injective.  Since $\smss(A_e^{em})$ bijects with $\twotilt(A_e^{em})$, by \cite[Cor 2.24]{Adachi}, $|\smss(A_e^{em})|={{2e}\choose{e}}$.  On the other hand, $\smss(A_e^e)$ bijects with the set of $\tau^e$-stable configurations of $\integer A_e$, which is the same as the set of configurations of $\integer A_e$.  Hence $\smss(A_e^e)$ bijects with non-crossing partitions of type $\mathbb{A}_e$, for which cardinality is well-known, namely the Catalan number $\frac{1}{e+1}{{2e}\choose{e}}$ (see, for example, \cite{Reading,S}).  This completes the proof.
\end{proof}

\textsc{Acknowledgement}  The author would like to thank Alexandra Zvonareva for discussions which leads to the creation of this work, and Takahide Adachi for explaining details of his results in \cite{Adachi} which helps to generalise the result in a previous draft.  Proposition \ref{prop-triangualtion-to-BTree} is obtained in a discussion with Takahide Adachi and Takuma Aihara during our collaboration for a forthcoming article.  The author is supported by EPSRC Doctoral Training Scheme.

\end{document}